\documentclass[letterpaper, 10 pt, conference]{ieeeconf}  % Comment this line out
\IEEEoverridecommandlockouts                              % This command is only
\usepackage[utf8]{inputenc}
\usepackage[T1]{fontenc}
\usepackage{graphicx}
\graphicspath{ {./images/} }
\usepackage{cite}
\usepackage{amsmath,amssymb,amsfonts,mathtools}
\usepackage{amsmath,amsfonts,amssymb}
\usepackage{epsfig,color} % for postscript graphics files
\usepackage{times} % assumes new font selection scheme 
\usepackage{amssymb}  % assumes amsmath package install
\usepackage{algorithm}
\usepackage{algpseudocode}
\usepackage{amsmath}
\usepackage{amsfonts}
\usepackage{graphicx}
\graphicspath{ {./images/} }
\usepackage{graphicx}
\usepackage{cuted}
\usepackage{graphics} % for pdf, bitmapped graphics files
\usepackage{epsfig,color} % for postscript graphics files
\usepackage{times} % assumes new font selection scheme 
\usepackage{amssymb}  % assumes amsmath package install
\usepackage{algorithm}
\usepackage{algpseudocode}
\usepackage{amsmath}
\usepackage{amsfonts}
\usepackage{graphicx}
\graphicspath{ {./images/} }
\usepackage{graphicx}
\usepackage{cuted}
\usepackage{booktabs} % for professional tables
\usepackage{subfig}

\usepackage{graphicx}
\usepackage{textcomp}
\usepackage{xcolor}
\def\BibTeX{{\rm B\kern-.05em{\sc i\kern-.025em b}\kern-.08em
    T\kern-.1667em\lower.7ex\hbox{E}\kern-.125emX}}

\newtheorem{theorem}{Theorem}
\newtheorem{lemma}{Lemma}
\newtheorem{assumption}{Assumption}
\newtheorem{remark}{Remark}

\newtheorem{proposition}{Proposition}

\title{\LARGE \bf  Persistent Mean Tracking in Client--Server Open Networks}

\author{Amit Dutta, Fat-Hy Omar Rajab, Marcos M. Vasconcelos  and Olugbenga M. Anubi  % <-this % stops a space
\thanks{A. Dutta, F. O. Rajab, M. M. Vasconcelos and O. M. Anubi are with the Department of Electrical and Computer Engineering, FAMU-FSU College of Engineering, Florida State University, Tallahassee, FL 32306, USA. E-mails:
        {\tt  \{adutta, frajab, m.vasconcelos, oanubi\}@fsu.edu}. }%
}

\begin{document}

\maketitle
\thispagestyle{empty}
\pagestyle{empty}

\begin{abstract}
% We study a client-server aggregation problem in open networks comprising two classes of agents: persistent agents whose time-varying mean we wish to track, and transient agents that must be filtered out. The main difficulty is that the server does not know the persistent subset a priori and observes only binary activity indicators together with scalar updates from active agents. To address this, we propose a window-based estimation framework that first identifies the persistent subset from repeated participation activity and then tracks its mean using a smoothing filter. Under a Bernoulli participation condition, we establish finite-window high-probability recovery guarantees for the persistent set and derive the corresponding formula for the minimum number of windows required. We then show that, after recovery, the tracking error contracts to a neighborhood determined by the target drift. We provide numerical simulations to illustrate the effectiveness of our theoretical results. 
We study persistent-agent identification and mean tracking in an open client-server network where agents participate intermittently. The network consists of a core of highly regular persistent agents and a subset of transient/non-persistent agents that appear only sporadically. Since the server observes only binary activity indicators and scalar updates from active agents, naive averaging cannot recover the active persistent mean because transient-agent updates contaminate the aggregate, and heterogeneous participation introduces bias. To address this, we propose a two-phase estimation framework. In Phase I, we develop a two-layer window-based identification procedure: the first layer forms single-window activity decisions, while the second aggregates these decisions across windows to create a high-probability separation between persistent and transient agents. Under heterogeneous Bernoulli participation, we derive explicit finite-window bounds that specify how many windows the server must wait for this separation to occur with high probability. In Phase II, we use the resulting structural estimate to track the active persistent mean via a smoothing filter, and show that the tracking error contracts geometrically to a neighborhood determined by the target drift and the filter gain. We provide numerical simulations to illustrate the effectiveness of our theoretical results.      
\end{abstract}
% \begin{IEEEkeywords}
% component, formatting, style, styling, insert
% \end{IEEEkeywords}
%%%%%%%%%%%%%%%%%%%%%%%%%%%% INTRODUCTION %%%%%%%%%%%%%%%%%%%%%%%%%%%%%%%%%%%%%%%%%%%%%%%
\section{Introduction}
\label{sec:introduction}

% Mean estimation and aggregation are fundamental primitives in networked control, distributed optimization, distributed estimation, and learning. In many modern applications, however, these tasks arise not over a peer-to-peer network, but in a \emph{client--server} architecture, where agents communicate only with a central server and do not exchange information directly among themselves. This setting is natural in server-assisted sensing, edge intelligence, and federated learning \cite{mcmahan2017communication,kairouz2021advances}. In this paper, we study mean estimation in an \emph{open} client--server network, where agent activity varies over time and the server does not know beforehand which agents constitute the stable long-term core of the system.
Mean estimation and aggregation are fundamental tasks in networked control, distributed optimization, estimation, and learning. In many modern applications, however, these tasks arise in a client--server architecture, where agents communicate only with a central server and not directly with one another. This setting is common in server-assisted sensing, edge intelligence, and federated learning \cite{mcmahan2017communication,kairouz2021advances}. In this paper, we study mean estimation in an open client--server network with time-varying agent activity, where the server does not know which agents form the stable long-term core of the system.

Specifically, among the \(N\) agents in \(\mathcal N:=\{1,2,\dots,N\}\), there exists a fixed but unknown subset \(\mathcal X\subset\mathcal N\) with \(|\mathcal X|=n\) whose members remain associated with the server over the full horizon and are active most of the time. We call these agents \emph{persistent}. They form the task-committed backbone of the network, although persistence does not imply uninterrupted availability: even persistent agents may become temporarily inactive due to recharging, communication outages, budget limitations, or temporary reassignment. The remaining \(m:=N-n\) agents are \emph{non-persistent}; for simplicity, we assume that \(m\) is the maximum size of this pool and that the active non-persistent participants vary within it through intermittent arrivals and departures. If \(\mathcal A_t\subseteq\mathcal N\) denotes the set of active agents at time \(t\), then the naive active-network mean is
\begin{align}
\label{eq:intro_naive_mean}
\bar{x}(t)
:=
\frac{1}{|\mathcal A_t|}\sum_{i\in\mathcal A_t}x_i(t).
\end{align}
In an open network, this quantity is generally not the right statistic, because it mixes persistent and non-persistent agents and therefore depends strongly on the transient composition of the visible set. We are instead interested in the persistent mean
\begin{align}
\label{eq:intro_persistent_mean}
x^\star(t)
:=
\frac{1}{|\mathcal X|}\sum_{i\in\mathcal X}x_i(t),
\end{align}
and, more directly, in the mean of the currently active persistent agents
\begin{align}
\label{eq:intro_active_persistent_mean}
x^\star_a(t)
:=
\frac{1}{|\mathcal X\cap\mathcal A_t|}
\sum_{i\in\mathcal X\cap\mathcal A_t}x_i(t).
\end{align}
Our goal is therefore to recover the unknown subset of persistent agents and track the active persistent mean \(x_a^\star(t)\), which is the quantity of interest, rather than the naive average \(\bar x(t)\).

This problem is not captured by standard aggregation models. Classical consensus and distributed averaging typically assume a fixed set of agents, or at least a known participating population, connected through a fixed or time-varying graph \cite{olfati2007consensus,jadbabaie2003coordination,ren2005consensus,nedic2009distributed}. Dynamic average consensus extends this setting to time-varying local signals \cite{kia2019dynamic}, but still presumes that the agents defining the target average are known. Likewise, distributed estimation methods over networked systems, including consensus-based estimators, focus on estimation over a known network structure rather than on recovering a hidden target-defining subset \cite{carron2014asynchronous}. In contrast, our setting is centralized rather than peer-to-peer, and the first challenge is not merely to estimate over a network, but to determine \emph{which agents should define the target mean in the first place}.

To make the distinction concrete, consider a simple simulation with two sets of agents: \(60\) persistent agents drawn from a Gaussian distribution centered at \(\mu_p=5\), and \(40\) non-persistent agents drawn from a different Gaussian distribution centered at \(\mu_{np}=0\). The average over the active persistent agents remains concentrated around the task-relevant persistent component, whereas a naive average over the overall population is shifted toward the mixed-population mean, as shown in Fig.~\ref{fig:intro_motivation}. This simple experiment shows that blind aggregation does not, in general, reflect the quantity of interest, and therefore motivates tracking the active persistent mean rather than an indiscriminate network-wide average.

\begin{figure}[t]
    \centering
    \includegraphics[width=\columnwidth]{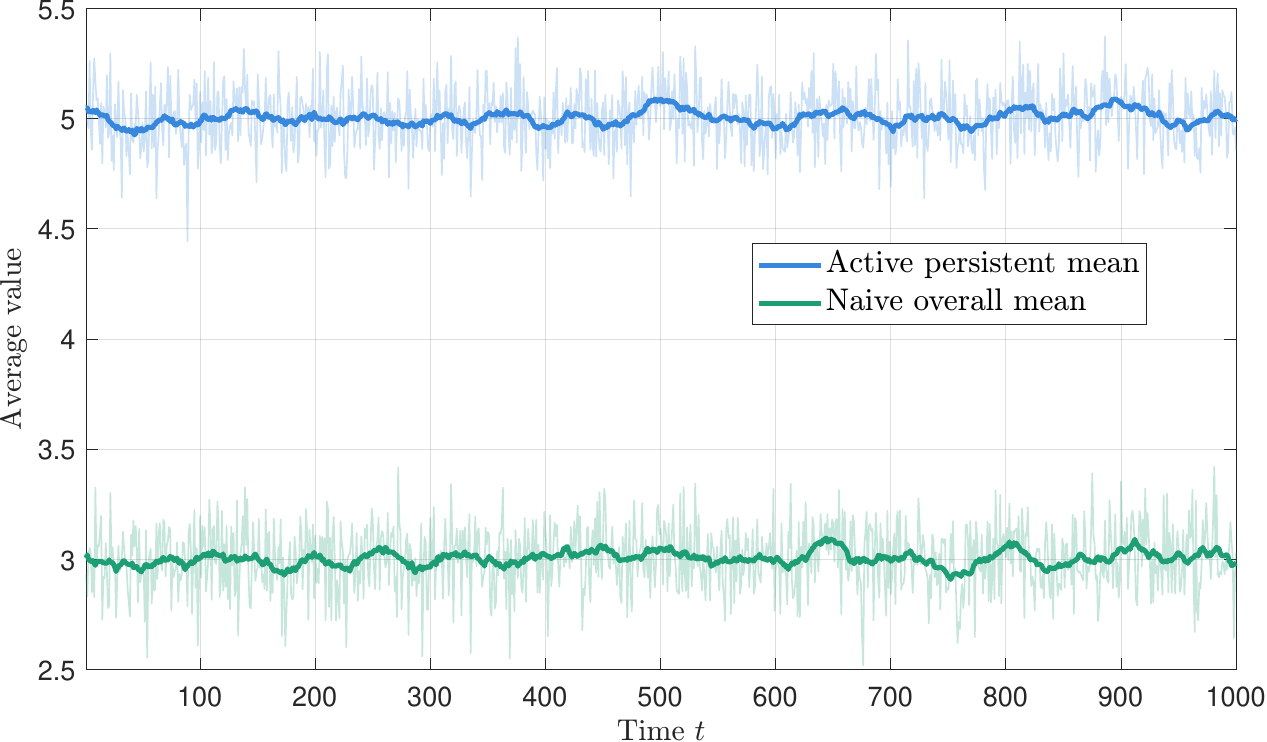}
    \caption{The blue curve shows the empirical average of the active persistent agents, which remains close to the persistent reference level. The green curve shows the naive overall average across the mixed agents, which is pulled toward the mixed-agents reference due to contamination from the non-persistent agents.}
    \label{fig:intro_motivation}
\end{figure}

The problem is related to \textit{open} multi-agent systems, where agents may join, leave, or be replaced over time \cite{hendrickx2017gossip,franceschelli2021stability,hadjicostis2024distributed,varma2018open,xue2022stability}. These works show that \textit{open}-ness fundamentally changes the behavior of consensus and averaging schemes. Related results include max-consensus under arrivals and departures \cite{abdelrahim2017max}, dynamic consensus in open systems \cite{franceschelli2018proportional,franceschelli2021stability}, and performance limitations for average consensus in open multi-agent systems \cite{monnoyer2023fundamental}. More recent works have also studied online optimization and resource allocation in open multi-agent systems \cite{vizuete2022resource} and graphon-based consensus models for open populations \cite{vizuete2025graphon}. However, these frameworks typically average, optimize, or reach agreement over the currently active or evolving population. Our formulation is different: the target-defining subset is fixed but unknown, and must itself be inferred from temporal activity patterns before one can meaningfully track its mean.

Federated learning provides the closest server-based viewpoint, since it also relies on a central coordinator and intermittently participating clients \cite{mcmahan2017communication,kairouz2021advances}. Recent works have shown that partial participation itself induces a distinct source of aggregation error \cite{jhunjhunwala2022fedvarp}, and open federated learning has begun to incorporate client churn explicitly \cite{sun2023openfl}. The same perspective is relevant for future client--server optimization and local-SGD-type methods, where naive aggregation over all active agents would induce a time-varying objective, whereas persistent-only aggregation points toward a stable core problem \cite{jhunjhunwala2022fedvarp,gupta2023byzantine,dutta2025open}. Still, even in these directions, the objective is ordinarily tied to participating clients rather than to an unknown stable core that must first be identified.

% The research gap is therefore not merely that the network is open, but that the \emph{target-defining subset is unknown}. Existing consensus, open-network optimization, and federated aggregation frameworks typically assume that the relevant population is either known or identified by current participation. Here, neither is true. The server does not know the persistent subset a priori, there is no peer-to-peer communication through which agents might self-organize around that subset, and naive aggregation over all visible agents yields a statistic that is generally unrelated to the desired persistent-core mean. This raises the central question of the paper: \emph{can a server recover a hidden persistent subset from repeated activity observations alone, and then use that recovered subset to track the correct task-relevant mean?}

% To address this question, we propose a two-stage framework that explicitly separates \emph{persistent-agent identification} from \emph{mean tracking}. In the first stage, the server uses repeated participation patterns over windows of length \(W\) to build a voting-based estimate of the persistent set. Participation counts are converted into binary window votes, and these votes are aggregated across windows to produce the recovered set
The main challenge is not only that the network is open, but that the target-defining subset is unknown. Existing consensus, open-network optimization, and federated aggregation frameworks typically assume that the relevant population is known or determined by current participation. Here, the server does not know the persistent subset a priori, and naive aggregation over active agents generally fails to recover the desired persistent-core mean. This raises the central question: can a server identify a hidden persistent subset from repeated activity observations alone and then use it to track the correct mean?

% To address this, we propose a two-stage framework separating persistent-agent identification from mean tracking. In the first stage, the server uses repeated participation over windows of length $W$ to construct a voting-based estimate of the persistent set, which it improves using an vote aggregating scheme $C_i(R)$ after Voting $R$ windows and form given threshold $\lamda$, it forms a high-confidence \textit{persitant} set estimate:
To address this, we propose a two-stage framework that separates persistent-agent identification from mean tracking. In the first stage, the server uses repeated participation over windows of length \(W\) to build a voting-based estimate of the persistent set. After \(R\) windows, it aggregates the window-level votes through the score \(C_i(R)\) and, using a threshold \(\lambda\), forms a high-confidence estimate of the \emph{persistent} set:
\begin{align}
\label{eq:intro_Xhat_compact}
\widehat{\mathcal X}(R):=\{\,i\in\mathcal N:\ C_i(R)\ge \lambda\,\}.
\end{align}
In the second stage, the server uses the recovered set to track the active persistent mean by smoothing the values of the currently active agents inside \(\widehat{\mathcal X}(R)\). Thus, unlike peer-to-peer consensus or distributed estimation algorithms \cite{olfati2007consensus,kia2019dynamic,carron2014asynchronous}, the proposed architecture is fully server-based and treats persistent-set recovery itself as part of the estimation problem.

The contributions of the paper are threefold. First, we formulate a mean-estimation problem for open client-server networks with an unknown fixed persistent subset. Second, we develop a window-based persistent-set recovery method together with a server-side tracking recursion for the active persistent mean. Third, we establish high-probability guarantees showing that the server can recover the persistent agents and then track their active mean. In this sense, our work differs from classical consensus and dynamic average consensus \cite{olfati2007consensus,kia2019dynamic} because the target-defining agent set is unknown; from open multi-agent averaging and optimization \cite{hendrickx2017gossip,franceschelli2021stability,monnoyer2023fundamental,vizuete2022resource} because the objective is not induced by the currently active set of agents; and from federated learning under partial participation \cite{mcmahan2017communication,kairouz2021advances,jhunjhunwala2022fedvarp,sun2023openfl} because the central task is not simply to aggregate intermittent client updates, but to identify and track a hidden persistent core. To the best of our knowledge, this persistent-subset mean-estimation problem has not been addressed in the existing literature.

%%%%%%%%%%%%%%%%%%%%%%%%%%%% RELATED WORK %%%%%%%%%%%%%%%%%%%%%%%%%%%%%%%%%%%%%%%%%%%%%%%
% \input{CDC 2026/Literature_new}

%%%%%%%%%%%%%%%%%%%%%%%% PROBLEM FORMULATION %%%%%%%%%%%%%%%%%%%%%%%%%%%%%%%%%%%%%%%%%%%
% \input{CDC 2026/Pr}
\section{Problem Formulation and Tracking Algorithm}
\label{sec:problem_formulation}

We consider a client-server system with agent set $\mathcal N:=\{1,2,\dots,N\}$. Time is discrete, indexed by $t=0,1,2,\dots$. At each time $t$, agent $i\in\mathcal N$ is either active or inactive, as indicated by the binary variable $s_i(t)\in\{0,1\}$. The corresponding active set is
\[
\mathcal A_t:=\{i\in\mathcal N:\ s_i(t)=1\}.
\]

Among the $N$ agents, there exists a fixed but unknown subset $\mathcal X\subset \mathcal N$, with $|\mathcal X|=n$, whose members remain associated with the task over the full horizon and are active most of the time. We refer to these agents as \emph{persistent}. The remaining agents form the complementary set $\mathcal Y:=\mathcal N\setminus \mathcal X$ with $|\mathcal Y|=m=N-n$, and are called \emph{non-persistent}. These agents may appear only sporadically, leave unpredictably, or otherwise participate without sustained regularity.

\paragraph*{Activity model}
For the analysis, we model each activity process $\{s_i(t)\}_{t\ge 0}$ as i.i.d.\ Bernoulli with parameter $p_i\in(0,1]$, where the probabilities $\{p_i\}$ are unknown to the server and may vary across agents. Persistent and non-persistent agents are distinguished statistically through their activity levels: persistent agents are assumed to be uniformly more active than non-persistent agents. Equivalently, there exist constants $0<\bar p_{\mathrm N}<\underline p_{\mathrm P}\le 1$ such that $p_i\ge \underline p_{\mathrm P}$ for $i\in\mathcal X$ and $p_j\le \bar p_{\mathrm N}$ for $j\in\mathcal Y$. This is the basic identifiability condition underlying the persistent-set recovery analysis; importantly, the server does \emph{not} know the individual values of $p_i$.

Each agent $i\in\mathcal N$ is associated with a scalar value $x_i(t)\in\mathbb R$, revealed to the server only when agent $i$ is active. The ideal target is the persistent mean $x^\star(t)$ in \eqref{eq:intro_persistent_mean}, while the directly visible quantity is the active persistent mean $x^\star_a(t)$ in \eqref{eq:intro_active_persistent_mean}. The difficulty is that the server does not know the persistent set $\mathcal X$. Consequently, a blind average over the active set $\mathcal A_t$ generally fails to recover either target, since $\mathcal A_t$ contains both persistent and non-persistent agents and, moreover, persistent agents may have heterogeneous activity rates.

\paragraph*{Window-vote identification}
To infer the latent set $\mathcal X$, the server monitors participation over windows of length $W\in\mathbb N$. For window index $r\in\mathbb N$, define the activity count of agent $i$ by
\begin{equation}
\label{eq:window_count}
K_{i,r}:=\sum_{u=0}^{W-1} s_i\big((r-1)W+u\big).
\end{equation}
Given a count threshold $k^\star\in\{0,1,\dots,W\}$, the corresponding window vote is
\begin{equation}
\label{eq:window_vote}
b_{i,r}:=\mathbf 1\{K_{i,r}\ge k^\star\}.
\end{equation}
Thus, agent $i$ receives a vote in window $r$ whenever it is active often enough within that window. After $R$ completed windows, the cumulative vote score is
\begin{equation}
\label{eq:cumulative_vote}
C_i(R):=\frac{1}{R}\sum_{r=1}^{R} b_{i,r}.
\end{equation}
Using a macro-threshold $\lambda\in(0,1)$, the server declares agent $i$ persistent if $C_i(R)\ge \lambda$, and hence the recovered persistent set is
\begin{equation}
\label{eq:recovered_persistent_set}
\widehat{\mathcal X}(R):=\{\,i\in\mathcal N:\ C_i(R)\ge \lambda\,\}.
\end{equation}

To quantify identification accuracy, let $\chi_i\in\{0,1\}$ denote the true persistence label, with $\chi_i=1$ if $i\in\mathcal X$ and $\chi_i=0$ otherwise, and let $\widehat\chi_i(R):=\mathbf 1\{C_i(R)\ge\lambda\}$. Similarly, we have the vectors $\chi, \hat{\chi}(R) \in \{0,1\}^{|\mathcal{N}|}$ to represent the ground truth and the estimated set of persistent agents, respectively, at any time. We measure the structural recovery error through the Hamming mismatch
\begin{equation}
\label{eq:Hamming_metric}
D(\hat{\chi}(R),\chi)=D(R):=\sum_{i\in\mathcal N}\big|\widehat\chi_i(R)-\chi_i\big|.
\end{equation}
The condition $D(R)=0$ is equivalent to exact recovery of the persistent set, i.e., $\widehat{\mathcal X}(R)=\mathcal X$.

\paragraph*{Tracking layer}
At time $t$, let $R(t):=\Big\lfloor \frac{t}{W}\Big\rfloor$ denote the number of completed windows. Based on the current estimate $\widehat{\mathcal X}(R(t))$, the natural eligible set for tracking is the subset of identified persistent agents that are also active at time $t$ given by $\mathcal E(t) = \{\widehat{\mathcal X}(R(t)) \cap \mathcal{A}_t\}$.
% \[
% \mathcal E(t):=\{\ell\in \widehat{\mathcal X}(R(t)):\ s_\ell(t)=1\}.
% \]
The server forms the active-normalized filtered mean
\begin{equation}
\label{eq:xtilde_active}
\tilde x(t):=
\begin{cases}
\dfrac{1}{|\mathcal E(t)|}\sum_{\ell\in\mathcal E(t)} x_\ell(t), & |\mathcal E(t)|>0,\\[1ex]
\hat x_t, & |\mathcal E(t)|=0,
\end{cases}
\end{equation}
and updates its running estimate through
\begin{equation}
\label{eq:xhat_update}
\hat x_{t+1}=(1-\eta_t)\hat x_t+\eta_t \tilde x(t),
\qquad \eta_t\in(0,1].
\end{equation}
Thus, the server does not wait until the persistent set is perfectly identified before beginning to track. Instead, it starts updating immediately using the current estimate $\widehat{\mathcal X}(R(t))$, allowing early mistakes and improving over time as more windows are observed.
\begin{algorithm}[t]
\caption{Window-Based Persistent Estimator}
\label{alg:window_persistent}
\begin{algorithmic}[1]
\State \textbf{Inputs:} window length \(W\), count threshold \(k^\star\), confidence threshold \(\lambda\), stepsizes \(\{\eta_t\}\)
\State \textbf{Initialize:} \(r\gets 0\), \(\widehat{\mathcal X}(0)\gets \emptyset\), \(C_i(0)\gets 0\) for all \(i\), \(K_i\gets 0\) for all \(i\), initial estimate \(\hat{x}_0\)

\For{$t=0,1,2,\dots$}
    \State Observe the active set \(\mathcal A_t:=\{i:\ s_i(t)=1\}\) and receive \(x_i(t)\) for all \(i\in\mathcal A_t\)

    \State Update the current-window participation counts:
    \[
    K_i \gets K_i + s_i(t), \qquad \forall i\in\mathcal N
    \]

    \If{$(t+1)\bmod W = 0$}
        \State \(r\gets r+1\)
        \For{each agent \(i\in\mathcal N\)}
            \State \(b_{i,r}\gets \mathbf{1}\{K_i\ge k^\star\}\)
            \State \(C_i(r)\gets \dfrac{r-1}{r}C_i(r-1)+\dfrac{1}{r}b_{i,r}\)
        \EndFor
        \State Update the recovered persistent set:
        \[
        \widehat{\mathcal X}(r)\gets \{i:\ C_i(r)\ge \lambda\}
        \]
        \State Reset \(K_i\gets 0\) for all \(i\in\mathcal N\)
    \EndIf

    \State Form the eligible set:
    \[
    \mathcal E(t)\gets \widehat{\mathcal X}(r)\cap \mathcal A_t
    \]

    \If{$|\mathcal E(t)|>0$}
        \State Compute the filtered mean:
        \[
        \tilde{x}(t)\gets \frac{1}{|\mathcal E(t)|}\sum_{i\in\mathcal E(t)} x_i(t)
        \]
    \Else
        \State Set \(\tilde{x}(t)\gets \hat{x}_t\)
    \EndIf

    \State Update the estimate:
    \[
    \hat{x}_{t+1}\gets (1-\eta_t)\hat{x}_t+\eta_t\tilde{x}(t)
    \]
\EndFor
\end{algorithmic}
\end{algorithm}

% \paragraph*{Problem statement}
% The objective of the paper is to design the identification parameters $(W,k^\star,\lambda)$ and the tracking recursion \eqref{eq:xhat_update} so that the server can:
% (i) recover the latent persistent set $\mathcal X$ with high probability from activity observations alone, and
% (ii) use the recovered set to track the active persistent mean $x_a^\star(t)$, and hence the persistent core of the network, with controlled error.

% The analysis therefore has two coupled parts. First, we establish finite-window recovery guarantees showing how the structural error $D(\hat{\chi},\chi)$ decreases as the number of windows $R$ grows. Second, we characterize how the quality of persistent-set recovery propagates into the tracking performance of $\hat x_t$. In particular, we identify a transition time $T_0=R_0W$ such that, after sufficiently many windows, the vectors $\hat{\chi}$ and $\chi$ align and thus, the event $\widehat{\mathcal X}(R)=\mathcal X$ holds with high probability for all $R\ge R_0$, and the tracking layer then operates on the correct stable core.
\paragraph*{Problem statement}
The objective is to design the identification parameters $(W,k^\star,\lambda)$ and the tracking recursion \eqref{eq:xhat_update} so that the server can:
(i) distinguish persistent agents from non-persistent agents with high probability using only activity observations, and
(ii) exploit this structural estimate to track the active persistent mean $x_a^\star(t)$, and hence the persistent core of the network, with controlled error.

Accordingly, the analysis has two coupled parts. First, we derive finite-window guarantees showing how the structural error $D(R)$ decays as the number of windows $R$ increases, and we obtain an explicit waiting time on the number of windows needed for high-probability separation. Second, we show how this separation accuracy propagates into the performance of the tracking recursion $\hat x_t$. In particular, we identify a transition time $T_0=R_0W$ beyond which the structural estimate is sufficiently accurate with high probability for the tracking layer to reliably follow the active persistent mean.

Thus, Phase I studies high-probability separation of persistent and non-persistent agents from activity observations, while Phase II studies how this separation feeds into dynamic tracking of \(x_a^\star(t)\).

\section{Phase 1: Multi-Layer Structural Identification}
\label{sec:identification}

Recall that due to the transient nature of the non-persistent agents, direct estimation of $x^{\star}(t)$ is not possible; instead, the server aims to compute $x^{\star}_{a}(t)$, which in turn requires prior identification of the persistent set $\mathcal{X}$. As the participation probabilities $p_i$ are unknown to the server, this identification must be inferred from the temporal activity patterns of the agents. 

We propose a two-layer strategy for identifying the persistent-agent set: Layer I performs a window-level coarse estimate, and Layer II refines it into a high-confidence estimate.

\subsection{Layer 1: The Micro-Level (Single Window)}

The server partitions discrete time into non-overlapping observation windows of length
\(W\in\mathbb N\). Define the window boundary times $t_r:=rW$ for $r=0,1,2,\dots,$
 and the \(r\)-th window $\mathcal W_r:=\{t_r,\dots,t_{r+1}-1\}.$

\paragraph{Server activity score model.}
For each agent \(\ell\), the server maintains the activity score
\begin{equation}
\label{eq:ewma_half}
a_\ell(t+1)=\frac12 a_\ell(t)+\frac12 s_\ell(t),
\qquad
a_\ell(0)\in[0,1].
\end{equation}
Unrolling \eqref{eq:ewma_half} gives
\begin{equation}
\label{eq:ewma_closed}
a_\ell(t)=2^{-t}a_\ell(0)+\sum_{j=0}^{t-1}2^{-(t-j)}s_\ell(j).
\end{equation}
Evaluating \eqref{eq:ewma_closed} at successive window boundaries yields
\begin{equation}
\label{eq:window_decomp}
a_\ell(t_{r+1})=2^{-W}a_\ell(t_r)+\psi_{\ell,r},
\end{equation}
where the current-window weighted contribution is
\begin{equation}
\label{eq:psi_def}
\psi_{\ell,r}:=
2^{-W}\sum_{u=0}^{W-1}2^u\,s_\ell(t_r+u).
\end{equation}
Here \(u\in\{0,\dots,W-1\}\) is the local index within window \(r\). For each agent, the server also maintains the unweighted activity count within the
current window:
\begin{equation}
\label{eq:count_def}
K_{\ell,r}
=
\sum_{t\in\mathcal W_r}s_\ell(t)
=
\sum_{u=0}^{W-1}s_\ell(t_r+u).
\end{equation}
Note that, the Bernoulli activity model, \(K_{\ell,r}\sim\mathrm{Binomial}(W,p_\ell)\). The server next applies an integer threshold \(k^\star\in\{1,\dots,W\}\) to these counts and defines the normalized threshold ratio \(\theta:=k^\star/W\). This gives the server an initial window-level rule for deciding which agents to vote as persistent.

\begin{proposition}[Activity score-to-count sufficient condition]
\label{prop:count_sufficient}
Let \(\tau\in(0,1-2^{-W}]\) and define
\begin{equation}
\label{eq:kstar_def}
k^\star:=
\left\lceil \log_2(\tau 2^W+1)\right\rceil.
\end{equation}
If \(K_{\ell,r}\ge k^\star\), then \(\psi_{\ell,r}\ge \tau\).
\end{proposition}

\begin{proof}
Fix a window \(r\) and an agent \(\ell\). Let
\[
A_{\ell,r}:=
\{u\in\{0,\dots,W-1\}: s_\ell(t_r+u)=1\},
\]
and let \(|A_{\ell,r}|=K_{\ell,r}=k\). From \eqref{eq:psi_def},
\[
\psi_{\ell,r}=2^{-W}\sum_{u\in A_{\ell,r}}2^u.
\]
For fixed \(k\), the right-hand side is minimized when the active indices are the
smallest \(k\) values, namely \(u=0,\dots,k-1\). Hence
\[
\sum_{u\in A_{\ell,r}}2^u
\ge
\sum_{u=0}^{k-1}2^u
=
2^k-1,
\]
so
\[
\psi_{\ell,r}\ge 2^{-W}(2^k-1).
\]
Thus \(\psi_{\ell,r}\ge\tau\) is guaranteed whenever $2^{-W}(2^k-1)\ge\tau,$
that is, whenever $k\ge \log_2(\tau 2^W+1).$ Therefore \(K_{\ell,r}\ge k^\star\) implies \(\psi_{\ell,r}\ge\tau\).
\end{proof}
\begin{remark}[Interpretation of \(k^\star/W\)]
Proposition~\ref{prop:count_sufficient} is stated for the score update
\[
a_\ell(t+1)=(1-\beta)a_\ell(t)+\beta s_\ell(t)
\quad\text{with}\quad \beta=\tfrac12,
\]
for which more recent activations receive larger weights than older ones. The threshold
\(k^\star\) is therefore a worst-case sufficient count: \(K_{\ell,r}\ge k^\star\)
guarantees \(\psi_{\ell,r}\ge\tau\) independently of the placement of activations
within the window, whereas more recent activations may require fewer than \(k^\star\).
Accordingly, the ratio \(\theta:=k^\star/W\) can be interpreted as a conservative
minimum window-level activation rate required for deterministic flagging within a
single window.

Since \(K_{\ell,r}/W\) is the empirical participation rate in window \(r\) with mean \(p_\ell\), the threshold \(\theta=k^\star/W\) serves as a decision boundary separating persistent and non-persistent agents at the level of their Bernoulli activity probabilities.
\end{remark}

The server next assigns a binary persistence vote to agent \(i\) in window \(r\) according to
\begin{align}
\label{eq:window_vote}
b_{i,r}=\mathbf{1}\{K_{i,r}\ge k^\star\}.
\end{align}
Thus, for the window-level classifier, the relevant derived quantity is the vote probability
\(\Pr(b_{i,r}=1)\), which is induced by the underlying participation probability \(p_i\). The following result formalizes the initial statistical separation induced by a single window: agents whose Bernoulli participation probabilities lie sufficiently above or below the threshold \(\theta\) receive opposite vote decisions with high probability.

\begin{lemma}[One-window vote separation]
\label{lem:one_window_separation}
Let \(\theta:=k^\star/W\) and fix \(\delta\in(0,1/2)\). Define
\begin{equation}
\label{eq:rho_bounds_compact}
\rho_{\mathrm P}:=
\theta+\sqrt{\frac{1}{2W}\log\frac{1}{\delta}},
\qquad
\rho_{\mathrm N}:=
\theta-\sqrt{\frac{1}{2W}\log\frac{1}{\delta}}.
\end{equation}
Assume \(0\le \rho_{\mathrm N}<\theta<\rho_{\mathrm P}\le 1\). If $p_i\ge \rho_{\mathrm P}\quad (i\in\mathcal X)$ and $p_j\le \rho_{\mathrm N}\quad (j\in\mathcal Y)$, then
\begin{equation}
\label{eq:one_window_vote_sep_compact}
\Pr(b_{i,r}=1)\ge 1-\delta,
\qquad
\Pr(b_{j,r}=1)\le \delta.
\end{equation}
For later use, define
\begin{equation}
\label{eq:piPN_def}
\pi_{\mathrm P}:=
1-e^{-2W(\rho_{\mathrm P}-\theta)^2},
\qquad
\pi_{\mathrm N}:=
e^{-2W(\theta-\rho_{\mathrm N})^2},
\end{equation}
so that
\[
\mathbb E[b_{i,r}] \ge \pi_{\mathrm P}\quad (i\in\mathcal X),
\qquad
\mathbb E[b_{j,r}] \le \pi_{\mathrm N}\quad (j\in\mathcal Y).
\]
\end{lemma}

\begin{proof}
Recall that $$K_{i,r}=\sum_{u=0}^{W-1}s_i(t_r+u),$$ with \(\{s_i(t_r+u)\}_{u=0}^{W-1}\) are i.i.d.\ Bernoulli\((p_i)\). Hence $K_{i,r}\sim \mathrm{Binomial}(W,p_i),$ and $\mathbb E\!\left[K_{i,r}/W\right]=p_i.$
Moreover, by definition of the vote rule, $b_{i,r}=1$ implies $K_{i,r}/W\ge \theta$. Thus the vote event is determined by whether the empirical participation rate in the
window crosses the threshold \(\theta\).

Consider first a persistent agent \(i\in\mathcal X\). Then
\[
\Pr(b_{i,r}=1)
=
\Pr\!\left(\frac{K_{i,r}}{W}\ge \theta\right)
=
1-\Pr\!\left(\frac{K_{i,r}}{W}<\theta\right).
\]
If \(p_i\ge \rho_{\mathrm P}>\theta\), then
\[
\Pr\!\left(\frac{K_{i,r}}{W}<\theta\right)
=
\Pr\!\left(\frac{K_{i,r}}{W}-p_i<-(p_i-\theta)\right).
\]
Since \(K_{i,r}/W\) is the empirical mean of \(W\) i.i.d.\ Bernoulli trials, Hoeffding's
inequality applies and gives
\begin{align*}
    \Pr\!\left(\frac{K_{i,r}}{W}-p_i<-(p_i-\theta)\right)
\le
e^{-2W(p_i-\theta)^2}
&\le
e^{-2W(\rho_{\mathrm P}-\theta)^2}.
\end{align*}
Therefore, $\Pr(b_{i,r}=1)\ge 1-\delta.$ Now consider a non-persistent agent \(j\in\mathcal Y\). Then
\[
\Pr(b_{j,r}=1)
=
\Pr\!\left(\frac{K_{j,r}}{W}\ge \theta\right)
=
\Pr\!\left(\frac{K_{j,r}}{W}-p_j\ge \theta-p_j\right).
\]
If \(p_j\le \rho_{\mathrm N}<\theta\), another application of Hoeffding's inequality yields
\[
\Pr\!\left(\frac{K_{j,r}}{W}-p_j\ge \theta-p_j\right)
\le
e^{-2W(\theta-p_j)^2}
\le
e^{-2W(\theta-\rho_{\mathrm N})^2}.
\]
Hence, $\Pr(b_{j,r}=1)\le \delta.$ Finally, since \(b_{i,r}\) and \(b_{j,r}\) are Bernoulli random variables,
\[
\mathbb E[b_{i,r}]=\Pr(b_{i,r}=1),
\qquad
\mathbb E[b_{j,r}]=\Pr(b_{j,r}=1),
\]
which gives the stated bounds involving \(\pi_{\mathrm P}\) and \(\pi_{\mathrm N}\).
\end{proof}
\begin{remark}[Interpretation and tradeoff]
The lemma shows that a single window already provides a probabilistic separation between
persistent and non-persistent agents. In particular, it yields the conservative sufficient
participation-probability gap
\[
\rho_{\mathrm P}-\rho_{\mathrm N}
=
2\sqrt{\frac{1}{2W}\log\frac{1}{\delta}},
\]
under which persistent and non-persistent agents receive opposite vote decisions with high
probability. Thus, increasing the window length \(W\) improves the statistical
separability of the two classes at the single-window level. However, a larger window also delays the server's decision, since the vote \(b_{i,r}\)
can only be formed after the entire window has been observed. Relying on a very long
window is therefore undesirable in a dynamic open network, where the server must begin
forming decisions before long observation horizons have elapsed.
\end{remark}

For this reason, the server does not wait for a single very long window to make a final
classification. Instead, it uses a finite window length \(W\) to form provisional
window-level decisions and then aggregates these decisions across multiple windows. In
this way, the server is allowed to make imperfect decisions initially and improve them
progressively as more windows are observed. The role of the second layer is precisely to
amplify the one-window vote gap over repeated observations and thereby produce a
high-confidence estimate of the persistent set after a finite number of windows.

\subsection{Layer 2: The Macro-Level (Multi-Window)}

In the second layer, the server aggregates the window votes over \(R\) observation windows:
\begin{align}
\label{eq:Cumilitive_window_vote}
C_i(R)=\frac{1}{R}\sum_{r=1}^R b_{i,r}.
\end{align}
Because non-overlapping windows are independent under the activity model, the votes
\(\{b_{i,r}\}_{r=1}^R\) are independent Bernoulli random variables whose expectations are
separated as in Lemma~\ref{lem:one_window_separation}. The server then applies a
macro-threshold \(\lambda\) and declares
\begin{equation}
\label{eq:Xhat_macro}
\widehat{\mathcal X}(R)
=
\{\,i\in\mathcal N \mid C_i(R)\ge \lambda\,\}.
\end{equation}
The associated indicator is
\begin{align}
\label{eq:chat_def}
\widehat{\chi}_i(R)=\mathbf{1}\{C_i(R)\ge \lambda\}.
\end{align}
Thus, the across-window classifier amplifies the one-window vote separation over repeated
observations.

\paragraph{Hamming recovery metric.}
Let the true persistence indicator be
\[
\chi_\ell=
\begin{cases}
1,& \ell\in\mathcal X,\\
0,& \ell\in\mathcal Y,
\end{cases}
\]
and let \(\widehat{\chi}_\ell(R)\) be defined as in \eqref{eq:chat_def}. We define the
Hamming identification error by
\begin{equation}
\label{eq:hamming_error}
D(R):=\sum_{\ell\in\mathcal N}\bigl|\widehat{\chi}_\ell(R)-\chi_\ell\bigr|.
\end{equation}
Thus, \(D(R)\) is exactly the number of misclassified agents after \(R\) windows. Note that since \(D(R)=0\) if and only if \(\widehat{\chi}_i(R)=\chi_i\) for all \(i\in\mathcal N\),
the event \(D(R)=0\) is equivalent to exact recovery of the persistent set. Next, we present our main result for Layer-II for estimating this persistent set with high probability after a finite number of observed windows.

\begin{theorem}[Finite-window recovery and sample complexity]
\label{thm:finite_window_recovery}
Suppose the conditions of Lemma~\ref{lem:one_window_separation} hold. Let $\lambda\in(\pi_{\mathrm N},\pi_{\mathrm P}),$ define the classification margins
\[
\Delta_{\mathrm P}:=\pi_{\mathrm P}-\lambda>0,
\qquad
\Delta_{\mathrm N}:=\lambda-\pi_{\mathrm N}>0,
\]
and let
\[
\Delta:=\min\{\Delta_{\mathrm P},\Delta_{\mathrm N}\}.
\]
Then the following statements hold for every \(R\ge 1\):

\begin{enumerate}
    \item The expected Hamming recovery error satisfies
    \begin{equation}
    \label{eq:EDR_bound_main}
    \mathbb E[D(R)]
    \le
    |\mathcal X|e^{-2R\Delta_{\mathrm P}^2}
    +
    |\mathcal Y|e^{-2R\Delta_{\mathrm N}^2}
    \le
    |\mathcal N|e^{-2R\Delta^2}.
    \end{equation}

    \item The exact recovery probability satisfies
    \begin{equation}
    \label{eq:exact_recovery_prob_main}
    \Pr\bigl(D(R)=0\bigr)
    \ge
    1-|\mathcal N|e^{-2R\Delta^2}.
    \end{equation}

    \item For any prescribed confidence level \(\rho\in(0,1)\), if
    \begin{equation}
    \label{eq:R0_recovery_main}
    R
    \ge
    R_0
    :=
    \frac{1}{2\Delta^2}
    \log\!\left(\frac{|\mathcal N|}{\rho}\right),
    \end{equation}
    then
    \begin{equation}
    \label{eq:one_minus_rho_recovery}
     \Pr\bigl(D(R)=0\bigr)\ge 1-\rho.
    \end{equation}
\end{enumerate}
\end{theorem}

\begin{proof}
We begin with the expected Hamming error. 
% By definition,
% \[
% D(R)=\sum_{i\in\mathcal N}\bigl|\widehat{\chi}_i(R)-\chi_i\bigr|.
% \]
For each persistent agent \(i\in\mathcal X\), a misclassification occurs when
\(C_i(R)<\lambda\), whereas for each non-persistent agent \(j\in\mathcal Y\), a
misclassification occurs when \(C_j(R)\ge \lambda\). Hence
\begin{equation}
\label{eq:EDR_split_main}
\mathbb E[D(R)]
=
\sum_{i\in\mathcal X}\Pr\!\big(C_i(R)<\lambda\big)
+
\sum_{j\in\mathcal Y}\Pr\!\big(C_j(R)\ge \lambda\big).
\end{equation}

Now fix \(i\in\mathcal X\). Since the votes \(\{b_{i,r}\}_{r=1}^R\) are i.i.d.\ Bernoulli
and Lemma~\ref{lem:one_window_separation} gives $\mu_i:=\mathbb E[b_{i,r}] \ge \pi_{\mathrm P},$
we have
\[
\Pr\!\big(C_i(R)<\lambda\big)
=
\Pr\!\big(C_i(R)-\mu_i < \lambda-\mu_i\big).
\]
Because \(\mu_i-\lambda\ge \pi_{\mathrm P}-\lambda=\Delta_{\mathrm P}\), Hoeffding's
inequality implies
\begin{equation}
\label{eq:FN_hamming_main}
\Pr\!\big(C_i(R)<\lambda\big)\le e^{-2R\Delta_{\mathrm P}^2}.
\end{equation}

Similarly, for \(j\in\mathcal Y\), Lemma~\ref{lem:one_window_separation} gives $\mu_j:=\mathbb E[b_{j,r}] \le \pi_{\mathrm N}.$ Thus
\[
\Pr\!\big(C_j(R)\ge \lambda\big)
=
\Pr\!\big(C_j(R)-\mu_j \ge \lambda-\mu_j\big),
\]
and since \(\lambda-\mu_j\ge \lambda-\pi_{\mathrm N}=\Delta_{\mathrm N}\), another
application of Hoeffding's inequality gives
\begin{equation}
\label{eq:FP_hamming_main}
\Pr\!\big(C_j(R)\ge \lambda\big)\le e^{-2R\Delta_{\mathrm N}^2}.
\end{equation}

Substituting \eqref{eq:FN_hamming_main} and \eqref{eq:FP_hamming_main} into
\eqref{eq:EDR_split_main} proves \eqref{eq:EDR_bound_main}. The simplified bound follows
from \(\Delta=\min\{\Delta_{\mathrm P},\Delta_{\mathrm N}\}\). For the exact recovery probability, note that \(D(R)\) is nonnegative and integer-valued,
so Markov's inequality gives
\[
\Pr\bigl(D(R)\ge 1\bigr)\le \mathbb E[D(R)].
\]
Using \eqref{eq:EDR_bound_main},
\[
\Pr\bigl(D(R)\ge 1\bigr)\le |\mathcal N|e^{-2R\Delta^2}.
\]
Since \(D(R)=0\) is equivalent to \(\widehat{\mathcal X}(R)=\mathcal X\), this proves
\eqref{eq:exact_recovery_prob_main}. Finally, if \(R\) satisfies \eqref{eq:R0_recovery_main}, then $|\mathcal N|e^{-2R\Delta^2}\le \rho,$ and \eqref{eq:one_minus_rho_recovery} follows directly from
\eqref{eq:exact_recovery_prob_main}.
\end{proof}

% \begin{remark}[Interpretation]
% The theorem is the main recovery result for Phase~I. It shows that the one-window vote
% gap established in Lemma~\ref{lem:one_window_separation} is amplified exponentially over
% repeated windows. The key quantity is the margin
% \[
% \Delta=\min\{\pi_{\mathrm P}-\lambda,\lambda-\pi_{\mathrm N}\},
% \]
% which measures how well the macro-threshold \(\lambda\) separates the persistent and
% non-persistent one-window vote expectations. Larger margins lead to faster decay of the
% recovery error as shown in \eqref{eq:EDR_bound_main}, while \eqref{eq:R0_recovery_main} gives an explicit bound on
% the number of windows required to achieve \((1-\rho)\)-exact recovery.
% \end{remark}
\begin{remark}[Interpretation]
The theorem is the main recovery result for Phase~I. It shows that the one-window vote
gap established in Lemma~\ref{lem:one_window_separation} is amplified exponentially over
repeated windows. The key quantity is the margin
\[
\Delta=\min\{\pi_{\mathrm P}-\lambda,\lambda-\pi_{\mathrm N}\},
\]
which measures how well the macro-threshold \(\lambda\) separates the persistent and
non-persistent one-window vote expectations. Larger margins lead to faster decay of the
recovery error, as shown in \eqref{eq:EDR_bound_main}, while \eqref{eq:R0_recovery_main}
gives an explicit bound on the number of windows required to achieve \((1-\rho)\)-exact
recovery. 

The lower bound in \eqref{eq:exact_recovery_prob_main} is valid for every
\(R\ge 1\), but it becomes informative once $|\mathcal N|e^{-2R\Delta^2}\le 1,$ or equivalently for large enough $R$ values. For smaller values of \(R\), the right-hand side of
\eqref{eq:exact_recovery_prob_main} may be negative and is therefore vacuous, although
the bound itself remains mathematically valid.
\end{remark}

\section{Phase II: Dynamic Mean Tracking After Recovery}
\label{sec:tracking}

We now analyze mean tracking after the persistent-set estimate from Phase~I has stabilized.
Let $\widehat{\mathcal X}:=\widehat{\mathcal X}(R_0)$ and $T_0:=R_0W,$ where \(R_0\) is chosen from the Phase-I recovery guarantee. For a prescribed tolerance
\(\rho\in(0,1)\), Phase~I ensures that
\begin{equation}
\label{eq:phase2_recovery_event}
\Pr(\mathcal G_{R_0})\ge 1-\rho,
\qquad
\mathcal G_{R_0}:=\{D(R_0)=0\}.
\end{equation}
Thus, on \(\mathcal G_{R_0}\), the server uses the correct persistent set throughout
Phase~II.

After time \(T_0\), the server forms the eligible set $\mathcal E(t):=\widehat{\mathcal X}\cap\mathcal A_t.$
Since the active persistent mean is only defined when \(\mathcal X\cap\mathcal A_t\neq\emptyset\),
we analyze the \emph{informative update times} after \(T_0\), i.e., those times for which
\(|\mathcal E(t)|>0\). Relabel these informative times by \(t=T_0,T_0+1,\dots\). On such
steps, the target is the active persistent mean $x_a^\star(t)$ from \eqref{eq:intro_active_persistent_mean}
% \begin{equation}
% \label{eq:xastar_phase2}
% x_a^\star(t):=
% \frac{1}{|\mathcal X\cap\mathcal A_t|}
% \sum_{\ell\in\mathcal X\cap\mathcal A_t}x_\ell(t),
% \end{equation}
and the server update is
\begin{equation}
\label{eq:xhat_phase2}
\hat x_{t+1}=(1-\eta)\hat x_t+\eta \tilde x(t),
\qquad \eta\in(0,1],
\end{equation}
where
\begin{equation}
\label{eq:xtilde_phase2}
\tilde x(t):=
\frac{1}{|\mathcal E(t)|}\sum_{\ell\in\mathcal E(t)}x_\ell(t).
\end{equation}

Let \(\mathcal F_t\) denote the server history up to informative time \(t\).

\begin{assumption}[Bounded drift]
\label{ass:bounded_drift_phase2}
There exists a constant \(\nu\ge 0\) such that, on \(\mathcal G_{R_0}\),
\begin{equation}
\label{eq:bounded_drift_phase2}
\|x_a^\star(t+1)-x_a^\star(t)\|\le \nu,
\qquad t\ge T_0.
\end{equation}
\end{assumption}

\begin{theorem}[Constant-gain tracking under bounded drift]
\label{thm:phase2_constant_drift}
Assume \eqref{eq:phase2_recovery_event} and Assumption~\ref{ass:bounded_drift_phase2}.
Then, on the event \(\mathcal G_{R_0}\), the informative-step tracking error
\[
e_t:=\hat x_t-x_a^\star(t)
\]
satisfies, for every \(t\ge T_0\),
\begin{equation}
\label{eq:phase2_pathwise_bound}
\|e_{t+1}\|^2
\le
(1-\eta)\|e_t\|^2+\frac{\nu^2}{\eta}.
\end{equation}
% Consequently, with $r_t:=\mathbb E\!\left[\|e_t\|^2\mid \mathcal G_{R_0}\right],$ one has
% \begin{equation}
% \label{eq:phase2_recursion_rt}
% r_{t+1}\le (1-\eta)r_t+\frac{\Delta^2}{\eta},
% \end{equation}
% and therefore
% \begin{equation}
% \label{eq:phase2_finite_time_bound}
% r_t
% \le
% (1-\eta)^{\,t-T_0}r_{T_0}
% +
% \frac{\Delta^2}{\eta^2}\Bigl(1-(1-\eta)^{\,t-T_0}\Bigr),
% \qquad t\ge T_0.
% \end{equation}
% In particular,
Consequently,
\begin{equation}
\label{eq:phase2_limsup}
\limsup_{t\to\infty}
\mathbb E\!\left[
\|e_t\|^2
\ \middle|\
\mathcal G_{R_0}
\right]
\le
\frac{\nu^2}{\eta^2}.
\end{equation}
Hence the conditional mean-square tracking error converges to a neighborhood of the
correct active persistent mean \(x_a^\star(t)\), and the neighborhood size is determined
by the target drift \(\nu\) and the gain \(\eta\).
\end{theorem}

\begin{proof}
Consider the time step \(t\ge T_0\) and the event \(\mathcal G_{R_0}\). Since
\(D(R_0)=0\) on this event, the recovered set is correct, and therefore $\mathcal E(t)=\mathcal X\cap\mathcal A_t,$ and $\tilde x(t)=x_a^\star(t).$ Hence \eqref{eq:xhat_phase2} becomes
\[
\hat x_{t+1}=(1-\eta)\hat x_t+\eta x_a^\star(t).
\]
Subtracting \(x_a^\star(t+1)\) from both sides gives
\begin{align*}
    e_{t+1}
=
(1-\eta)e_t-\nu,
\qquad
\nu:=x_a^\star(t+1)-x_a^\star(t).
\end{align*}
Taking norms on both sides yields
\[
\|e_{t+1}\|^2
=
(1-\eta)^2\|e_t\|^2+\|\nu\|^2
-2(1-\eta)\langle e_t,\nu\rangle.
\]
Using Young's inequality $2|\langle e_t,\nu\rangle|
\le
\alpha\|e_t\|^2+\frac{1}{\alpha}\|\nu\|^2$ with the choice \(\alpha=\eta\), we obtain
\[
2(1-\eta)|\langle e_t,\nu\rangle|
\le
\eta(1-\eta)\|e_t\|^2+\frac{1-\eta}{\eta}\|\nu\|^2.
\]
Therefore, we have
\begin{align*}
\|e_{t+1}\|^2
&\le(1-\eta)\|e_t\|^2+\frac{\|\nu\|^2}{\eta}.
\end{align*}
Applying Assumption~\ref{ass:bounded_drift_phase2} gives \eqref{eq:phase2_pathwise_bound}. Now define $r_t:=\mathbb E\!\left[\|e_t\|^2\mid \mathcal G_{R_0}\right]$. Taking conditional expectation in \eqref{eq:phase2_pathwise_bound} gives
\[
r_{t+1}\le (1-\eta)r_t+\frac{\nu^2}{\eta},
\]
unrolling which for $t\ge T_0$ gives us
\begin{align*}
    r_t
\le
(1-\eta)^{\,t-T_0}r_{T_0}
+
\frac{\nu^2}{\eta^2}\Bigl(1-(1-\eta)^{\,t-T_0}\Bigr),
\end{align*}
and letting \(t\to\infty\) gives
\eqref{eq:phase2_limsup}.
\end{proof}

% \begin{remark}[On the Young-inequality choice]
% More generally, Young's inequality gives, for any \(\alpha>0\),
% \[
% \|e_{t+1}\|^2
% \le
% \Bigl((1-\eta)^2+\alpha(1-\eta)\Bigr)\|e_t\|^2
% +
% \left(1+\frac{1-\eta}{\alpha}\right)\|\Delta_t\|^2.
% \]
% The choice \(\alpha=\eta\) used above yields the clean recursion
% \[
% \|e_{t+1}\|^2\le (1-\eta)\|e_t\|^2+\frac{\|\Delta_t\|^2}{\eta},
% \]
% which is contractive for every \(\eta\in(0,1]\) and gives the smallest steady-state upper
% bound within this Young-family of estimates.
% \end{remark}

\begin{remark}[Interpretation and scope]
The theorem establishes that, after persistent-set recovery, a simple constant-gain recursion tracks the correct active persistent mean up to a neighborhood whose size depends on the drift bound $\nu$ and becomes smaller as the gain $\eta$ increases. Thus, the remaining Phase-II error is driven by target motion rather than by identification error.

A different regime arises when the drift scales with the step size, e.g.,
\[
\|x_a^\star(t+1)-x_a^\star(t)\|\le \eta_t \nu.
\]
In that case, a diminishing-gain choice such as \(\eta_t=\mathcal O(1/t)\) can yield a
shrinking mean-square error neighborhood. However, that assumption ties the target motion
to the filter gain and is not the focus of the present paper. If the drift is bounded
independently of the gain, as assumed here, then a constant-gain tracker is the more
natural choice. More general dynamic-tracking architectures for faster target motion,
for example using two-time-scale or adaptive designs, are left for future work.
\end{remark}

\begin{figure}[t]
    \centering
    \begin{minipage}{0.48\textwidth}
        \centering
        \includegraphics[width=\linewidth]{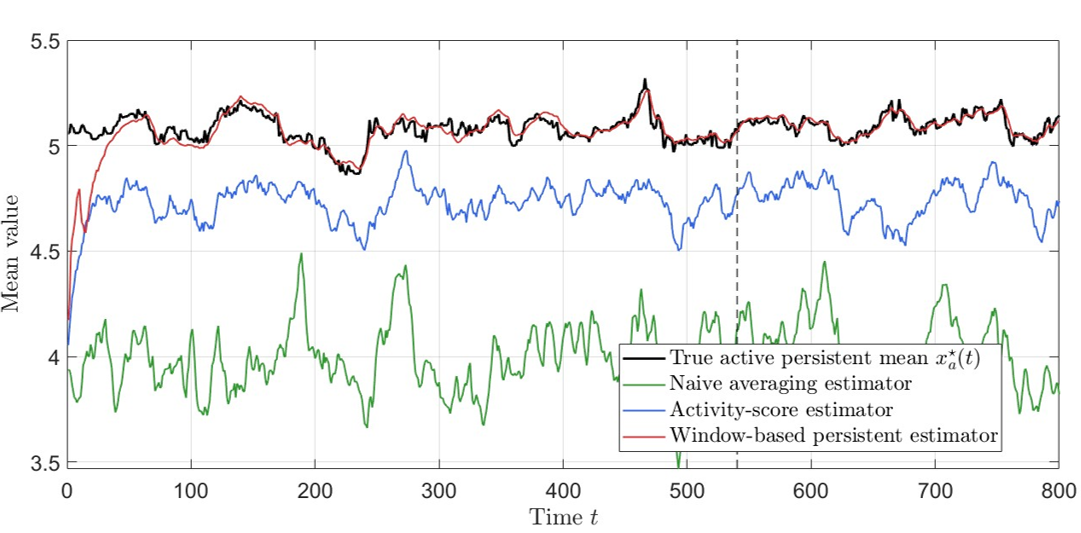}
        \caption{Tracking of the active persistent mean \(x_a^\star(t)\) by the naive averaging estimator, the activity-score estimator, and the proposed window-based persistent estimator.}
        \label{fig:mean_tracking}
    \end{minipage}
    \hfill
    \begin{minipage}{0.49\textwidth}
        \centering
        \includegraphics[width=\linewidth]{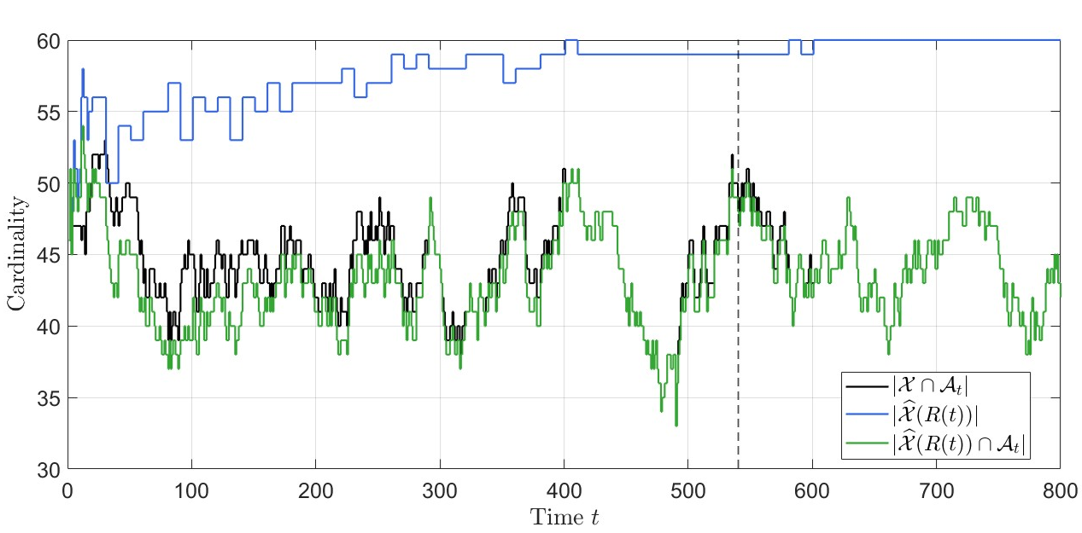}
        \caption{Cardinality of the true active persistent set \( |\mathcal X\cap \mathcal A_t| \), the recovered persistent set \( |\widehat{\mathcal X}(R(t))| \), and the recovered active intersection \( |\widehat{\mathcal X}(R(t))\cap \mathcal A_t| \). The vertical dashed line marks the theoretical recovery horizon \(T_0=540\).}
        \label{fig:set_sizes}
    \end{minipage}
\end{figure}

\section{Numerical Simulation}
We illustrate the two mechanisms predicted by the analysis: finite-window recovery of the persistent set and the resulting improvement in tracking the target mean. We simulate a network with \(N=100\) agents, of which \(n=60\) are persistent and \(m=40\) are non-persistent agents. Persistent agents are generated around \(\mu_{\mathrm P}=5\) with standard deviation \(0.8\), while non-persistent agents are generated around \(\mu_{\mathrm N}=0\) with standard deviation \(0.2\). The simulation horizon is \(T=800\), and all three estimators use the same tracking gain \(\eta=0.05\).

 Persistent agents participate with probability \(p_{\mathrm P}=0.75\), whereas non-persistent agents participate with probability \(p_{\mathrm N}=0.25\). For the proposed estimator, we use window length \(W=20\), vote threshold \(k^\star=10\), corresponding normalized threshold \(\theta=0.5\), and macro-threshold \(\lambda=0.5\). Choosing one-window confidence parameter \(\delta= 0.125\), \eqref{eq:rho_bounds_compact} gives \(\rho_{\mathrm P}\approx 0.73\) and \(\rho_{\mathrm N}\approx 0.27\). Hence \(p_{\mathrm P}\ge \rho_{\mathrm P}\) and \(p_{\mathrm N}\le \rho_{\mathrm N}\), satisfying Lemma~\ref{lem:one_window_separation}. Using Theorem~\ref{thm:finite_window_recovery}, the resulting margins \(\Delta_{\mathrm P}\) and \(\Delta_{\mathrm N}\) yield \(\Delta\approx 0.375\). For confidence level \(1-\rho=0.95\), \eqref{eq:R0_recovery_main} gives
\begin{equation}
    R_0
    =
    \left\lceil
    \frac{1}{2(0.375)^2}
    \log\!\left(\frac{100}{0.05}\right)
    \right\rceil
    =
    27
    \ \text{windows},
\end{equation}
which corresponds to the theoretical recovery horizon \(T_0=R_0W=540\) time steps.

% Fig.~\ref{fig:set_sizes} shows the structural effect of this recovery process. Early in the simulation, the recovered set is still adapting. As window votes accumulate, however, the recovered persistent set \( \widehat{\mathcal X}(R(t)) \) stabilizes near the true persistent cardinality \(n=60\), and its active intersection \( \widehat{\mathcal X}(R(t))\cap \mathcal A_t \) becomes closely aligned with the true active persistent set \( \mathcal X\cap \mathcal A_t \). This transition occurs around the predicted horizon \(T_0=540\), shown by the vertical dashed line.

% Fig.~\ref{fig:mean_tracking} shows the corresponding effect on the estimation layer. Once persistent-set recovery has stabilized, the proposed window-based persistent estimator tracks \(x_a^\star(t)\) substantially more accurately than the naive averaging estimator, which remains biased by contamination from the non-persistent population. The activity-score estimator provides a partial correction, but remains less accurate than the window-based method. Overall, the simulation is consistent with the finite-window theory: recovery of the persistent set around the predicted horizon is accompanied by a clear improvement in tracking accuracy.
Fig.~\ref{fig:set_sizes} shows the structural effect of the recovery process. As window votes accumulate, the recovered persistent set \( \widehat{\mathcal X}(R(t)) \) stabilizes near the true persistent cardinality \(n=60\), and its active intersection \( \widehat{\mathcal X}(R(t))\cap \mathcal A_t \) closely tracks the true active persistent set \( \mathcal X\cap \mathcal A_t \). This transition occurs around the predicted horizon \(T_0=540\), marked by the vertical dashed line.

Fig.~\ref{fig:mean_tracking} shows the resulting effect on estimation. After persistent-set recovery stabilizes, the proposed window-based persistent estimator tracks \(x_a^\star(t)\) significantly more accurately than the naive averaging estimator, while the activity-score estimator remains intermediate between the two. These results are consistent with the finite-window theory: recovery of the persistent set near the predicted horizon leads to a clear improvement in tracking accuracy.

\section{Conclusion}

This paper studied mean tracking in an open client--server network with an unknown persistent subset of agents. We showed that naive aggregation over the active agents generally fails to recover the task-relevant mean because non-persistent agents contaminate the active set. To address this, we proposed a two-stage framework: a finite-window voting scheme to recover the persistent set, followed by a server-side recursion to track the active persistent mean. The analysis established explicit high-probability guarantees for persistent-set recovery and a bounded-drift tracking result for the post-recovery phase. Together, these results show that a server can progressively identify a hidden stable core from repeated activity observations and then track its associated mean with controlled error.
\bibliographystyle{IEEEtran}
\bibliography{IEEEfull,my_bib}

\end{document}